\documentclass[letterpaper, 10 pt, conference]{ieeeconf}

\IEEEoverridecommandlockouts               
\renewcommand{\baselinestretch}{0.99}
\usepackage{url}
\usepackage[T1]{fontenc}
\usepackage{amssymb,amsmath,epsfig}
\usepackage[latin5]{inputenc}
\usepackage{graphicx,times,latexsym}
\usepackage{subfigure}
\usepackage{psfrag,cite}
\usepackage{setspace}
\usepackage{algorithmic, algorithm}
\usepackage{color}

\DeclareMathOperator*{\Tr}{tr}

\newcommand{\norm}[1]{\| #1 \|}

\newtheorem{theorem}{Theorem}

\newtheorem{remark}{Remark}

\newtheorem{ansatz}{Assumption}
\newcommand{\oprocendsymbol}{\hbox{$\bullet$}}
\newcommand{\oprocend}{\relax\ifmmode\else\unskip\hfill\fi\oprocendsymbol}

\newcommand{\floor}[1]{\lfloor #1 \rfloor}
\newcommand{\ceil}[1]{\lceil #1 \rceil}

\newcommand{\real}{\mathbb{R}}
\newcommand{\integers}{\mathbb{N}}

\graphicspath{{epsfiles/}}

\title{Time-triggering versus event-triggering control over communication channels}

\author{Mohammad Javad Khojasteh \; Pavankumar Tallapragada \; Jorge
  Cort{\'e}s \; Massimo Franceschetti \thanks{M. J. Khojasteh and
    M. Franceschetti are with the Department of Electrical and
    Computer Engineering of University of California, San Diego.
    J. Cort{\'e}s is with the Department of Mechanical and Aerospace
    Engineering, University of California, San Diego. P. Tallapragada
    is with the Department of Electrical Engineering, Indian Institute
    of Science, Bengaluru.  \{mkhojasteh,massimo,cortes\}@ucsd.edu,
    pavant@ee.iisc.ernet.in}}
\begin{document}


\maketitle

\begin{abstract}
Time-triggered and event-triggered control strategies for stabilization of an unstable plant over a rate-limited communication channel subject to unknown, bounded  delay are studied and compared. Event triggering carries implicit information, revealing the state of the plant. However,  
the delay in the communication channel causes information loss, as it makes the state information out of date.
There is a critical delay value,
 when  the loss of information due to the communication delay perfectly compensates the implicit information carried by the triggering events. This  occurs when the  maximum delay  equals the inverse of the \emph{entropy rate} of the plant.
In this context, extensions of our previous results for event triggering strategies are presented for vector systems and are compared with the    data-rate theorem  for time-triggered control, that is extended here to a  setting with unknown delay. 
\end{abstract}

\section{Introduction}\label{sec:intro}


Internet of things  establishes a foundation for emerging of engineering systems that integrate
computing, communication, and control, these systems are know as cyber-physical systems (CPS)~\cite{kumar,carruthers2014internet}.
One key aspect of CPS is the presence of finite-rate, digital communication channels in the feedback loop. To quantify their effect on the ability to stabilize the system, \emph{data-rate theorems} have been developed
~\cite{Mitter,nair2004stabilizability}.
They essentially state that, in order to achieve stabilization, the communication rate available in the feedback loop should be at least as large as the \emph{entropy rate} of the system, corresponding to the sum of the logarithms of the unstable modes. In this way, the controller  can compensate for the expansion of the state occurring during communication 

More recent formulations of data-rate theorems include stochastic, time-varying, Markovian, erasure, additive white and
colored Gaussian, and  multiplicative noise feedback communication channels~\cite{martins2006feedback,Paolo,Lorenzo,sukhavasi2016linear,ardestanizadeh2012control,ding2016multiplicative},   formulations for nonlinear sytems~\cite{de2005n,liberzon2009nonlinear, topological}, and for systems with uncertain and variable parameters~\cite{kostina2016rate,ling2010necessary,ranade2015control,nair2002communication}. Connections with information theory are highlighted in~\cite{topological,sahai2006necessity, matveev2009estimation, massimopaolo2016,girish2013}. Extended surveys of the literature appear in~\cite{Massimo} and~\cite{Nair}.
 
Another important aspect of CPS is the need to use distributed
resources efficiently. In this context, event-triggering control
techniques~\cite{Tabuada,WPMHH-KHJ-PT:12} have emerged. These are
based on the idea of sending information in an opportunistic manner
between the controller and the plant. In this way, communication
occurs only when needed, and the primary focus  is on minimizing the
number of transmissions while guaranteeing the control
objectives. Some recent results on event-triggered implementations
in the presence of data rate constraints appear in
\cite{PT-JC:16-tac,Level,ling2016bit,pearson2016control}.
One important observation raised in~\cite{Level} is  that using    event-triggering   is possible to ``beat''  the data-rate theorem.   More precisely,   if the channel does not introduce any delay, then an event-triggering strategy can achieve stabilization for any positive rate of transmission.
This apparent contradiction is resolved by realizing that the timing of the triggering events  carries information, revealing the state of the system. When communication occurs without delay, the state can be tracked with arbitrary precision, and transmitting a single bit  at every triggering event is enough to compute the appropriate control action. 

In our previous work \cite{khojasteh2016value}, we extended the above observation to the whole spectrum of possible delay values.
Key to our analysis was the distinction between the \emph{information access rate}, that is the rate at which the controller needs to receive data, regulated by the classic data-rate theorem; and the \emph{information transmission rate}, that is the rate at which the sensor needs to send data, regulated by a given triggering control strategy. For a given triggering strategy, we showed that  for sufficiently low values of the delay, the timing information carried by the triggering events is large enough and the system can be stabilized with any positive   information transmission rate. At a critical value of the delay, the timing information carried by event triggering is not enough for stabilization and the required information transmission rate begins to grow. When the delay reaches the inverse of the entropy rate of the plant, the timing information becomes completely obsolete, and the required information transmission rate becomes larger than the information access rate imposed by the data-rate theorem. 

In the present work, we compare these results with those of a
time-triggered implementation, for which we provide a formulation of
the data-rate theorem for continuous-time systems in the presence of
delay. The comparison leads to additional insights on the value of
information in event triggering. We also extend results
in~\cite{khojasteh2016value} to vector systems. 
Proofs of the results on
event-triggering 
are omitted and can be found in~\cite{OurJournal1}. 


\subsubsection*{Notation}
Let $\real$ and $\integers$ denote  the set of real and positive integer numbers, respectively.
We denote by $\mathcal{B}(\epsilon)$ the ball centered at $0$ of
radius~$\epsilon$. We let $\log$ and $\ln$ denote the logarithm with
bases $2$ and $e$, respectively. For any function $f : \real \rightarrow
\real^n$ and $t \in \real$, we let $f(t^+)$ denote the limit from the
right, namely $\lim_{s \downarrow t} f(s)$.  We let $M_{n,m}(\real)$
be the set of $n \times m$ matrices over the field of real
numbers. Let $0_{n}$ be the all $0$ vector of size $n$. Given $A=[a_{i,j}]_{1\le i,j \le n} \in M_{n,n}(\real)$, we
let $\Tr(A)=\sum_{i=1}^{n} a_{ii}$ and $\det(A)$ denote its trace and
determinant, respectively. Note that $\det(e^A)$ is equal to
$e^{\Tr(A)}$. We let $m$ denote the Lebesgue measure on
$\mathbb{\real}^{n}$, which for $n=2$, and $n=3$ corresponds to area
and volume, respectively. Note that for $A \in M_{n,n}(\real)$ and $X
\in \real^{n}$, $m(AX)=|\det(A)|m(X)$. We let $\floor{x}$ denote the
greatest integer less than or equal to $x$. 
We let $\norm{x}$ be the  $L^2$ norm of   $x$ in $\real^n$.

\section{Problem formulation}\label{sec:setup}



\subsubsection*{System model}

We consider a networked control system composed by the
plant-sensor-channel-controller tuple depicted in Figure~\ref{fig:system}.
\begin{figure}
	\centering
 \includegraphics[scale=0.5]{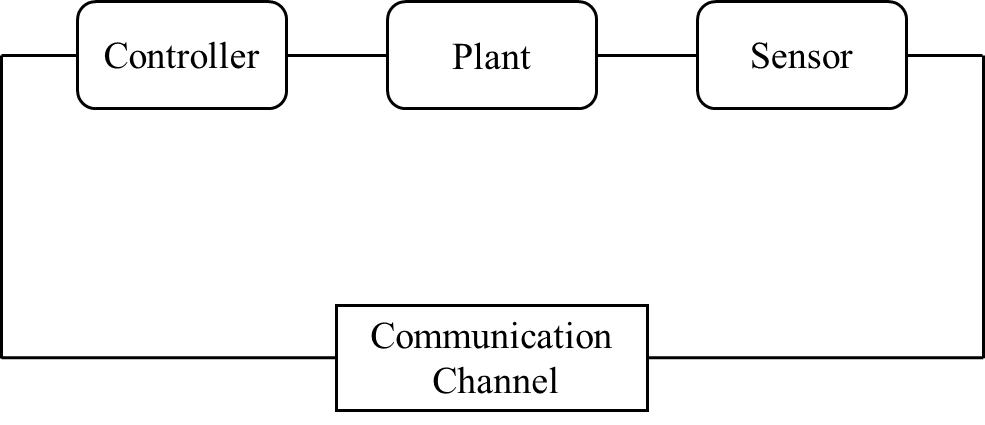}
	\caption{System model.}\label{fig:system}
\end{figure}
The plant dynamics are described by a vector, continuous-time, linear
time-invariant (LTI) system
\begin{align}\label{syscon}
  \dot{x}=Ax(t)+Bu(t),
\end{align}
where $x(t) \in \real^n$ and $u(t) \in \real^m$ for $t \in [0,\infty)$
are the plant state and control input, respectively. Here, $A \in
M_{n,n}(\mathbb{R})$, $B \in M_{n,m}(\mathbb{R})$, and
\begin{align*}
  x(0)\in \mathcal{B}(L)
\end{align*}
for some non-negative real number~$L$ ($L$ is known to both sensor and controller). Without loss of generality, we
assume that all the eigenvalues of $A$ are unstable, that is,
$\text{Re} \{\lambda_i\}>0$ for $i \in \{1,\ldots,n\}$. 
The sensor can measure the state
of the system exactly, and the controller can apply the control input
to the plant with infinite precision and without delay. However,
sensor and controller communicate through a channel that can support
only a finite data rate and is subject to delay, as we describe next.

  
\subsubsection*{Triggering Times and Communication
  Delay}\label{subsec:Delay}
We denote by $\{ t_s^k \}_{k \in \integers}$ the sequence of
times at which the sensor transmits a packet composed of
$g(t_s^k)$ bits representing the system state to the controller.
We
define the $k^{th}$ \emph{triggering interval} by
\begin{align}\label{gammma111}
  \Delta'_k = t_s^{k+1}-t_s^{k}.
\end{align}
We let $t_c^k $ be the time at
which the controller receives and decodes a packet of data which was encoded and transmitted at time $t_s^k$ for $k \in \integers$. We
assume a uniform upper bound, known to the sensor and the controller,
on the \emph{communication delays} 
\begin{align}\label{gammma}
  \Delta_k= t_c^k-t_s^k \leq \gamma.  
\end{align}
When
referring to a generic triggering time or reception time, we shall
skip the superscript $k$ in $t_s^k$ and $t_c^k$. 

\subsubsection*{Time-Triggered and Event-Triggered Control}

With the information received from the sensor, the controller
maintains an estimate $\hat{x}$ of the plant state, which during the inter-reception times evolves
according to
\begin{align}\label{sysest}
  \dot{\hat{x}}(t)=A\hat{x}(t)+Bu(t),
\end{align}
starting from $\hat{x}(t_c^{k+})$.

The \emph{state estimation error} is then
\begin{align}\label{eq:ese}
  z(t)=x(t)-\hat{x}(t),
\end{align}
 where $z(0)=x(0)-\hat{x}(0)$. Without updated information from the sensor, this
error grows, and the system can potentially become unstable.
The sensor should therefore select the sequence of transmission times
$\{ t_s^k \}_{k \in \integers}$ and the packet sizes $\{g(t^k_s)\}_{k \in \integers}$ in a way that ensures stabilizability and observability, while satisfying the rate constraints imposed by the channel.  

The asymptotic notions of stabilizability and observability that we require are standard, and are formally defined in \cite{Mitter,khojasteh2016value}.
To ensure these asymptotic properties, we consider two different approaches:
event-triggered and time-triggered control.
In an
event-triggering implementation, we   define a triggering
function $v(t)$
%
%
that is known to both the
controller and the sensor. Whenever the state estimation error crosses the value of this function, a transmission occurs. In a time-triggered implementation, transmissions are not state dependent. 
\subsubsection*{Information Access Rate}
We let $b_c(t)$ denote the number of bits that have
been received by the controller up to time $t$.
We define the \emph{information access rate}
\begin{align}
  R_c = \limsup_{t \rightarrow \infty} \frac{b_c(t)}{t}.
\label{infaccessrate}
\end{align}

In this setting, data-rate theorems
describe the trade-off between the information access rate and the
ability to stabilize the system. 
They are generally stated for
discrete-time systems, albeit similar arguments hold in continuous
time as well, see e.g.~\cite{hespanha2002towards}. They are based on the
fundamental observation that there is an inherent entropy rate
\begin{align}
  h(A)=\frac{\Tr(A)}{\ln 2} = \frac{\sum_{i=1}^d \lambda_i}{\ln 2}
\label{eq:inh}
\end{align}
at which the plant generates information. It follows 
that to guarantee stability it is necessary for the controller to have
access to state information at a rate
\begin{align}\label{Datarate}
  R_c > h(A).
\end{align}
This result indicates what is required by the controller, and it does
not depend on the feedback structure --- including aspects such as
communication delays, information pattern at the sensor and the
controller, and whether the times at which transmissions occur are
state dependent, as in event-triggered control, or not, as in
time-triggered control.

\subsubsection*{Information Transmission Rate}

We now take the viewpoint of the sensor when examining the amount of
information that it needs to transmit to the controller.  We make the following two observations. First,
in the presence of communication delays, the state estimate received
by the controller might be out of date, so that the sensor might need
to send data at a \emph{higher rate} than what~\eqref{Datarate}
prescribes to make-up for such discrepancy.  Second, in the case of
event-triggered transmissions, the timing of the triggering events
itself carries some information. For instance, if the communication
channel does not introduce any delay, then a triggering event may
reveal the state of the system very precisely, and effectively carry
an unbounded amount of information. The controller may then be able to
stabilize the system even if the sensor uses the channel very
sparingly, transmitting at a \emph{smaller rate} than
what~\eqref{Datarate} prescribes.

Motivated by these observations, let $b_s(t)$ be the number of bits
transmitted by the sensor up to time $t$, and define the
\emph{information transmission rate} by
\begin{align}\label{Tx-rate}
  R_s = \limsup_{t \rightarrow \infty}\frac{b_s(t)}{t}.
\end{align}
Since at every triggering time the sensor sends $g(t_s^k)$ bits, we also have
\begin{align}\label{rss}
  R_s = \limsup_{N \rightarrow \infty} \frac{\sum_{k=1}^{N} g(t_s^k)}{\sum_{k=1}^{N}
    \Delta'_k}.
\end{align}

\section{Necessary condition on the access rate for event triggering}\label{sec:Necessary} 

We now quantify the amount of information that the
controller needs to have access to in order to 
have exponential convergence of the estimation error and the
plant state to zero, irrespective of  the feedback
structure used by the sensor to decide when to transmit.  The proof  follows, with minor modifications to the one for the scalar case, cf.~\cite{OurJournal1}. 

\begin{theorem}\label{thm:necc-access-rate}
Consider the plant-sensor-channel-controller model described in
  Section~\ref{sec:setup} with plant dynamics~\eqref{syscon}, and state estimation error $z (t)$.  Let $\sigma \in \real$ be
  positive.
  \begin{enumerate}
  \item If the state estimation error satisfies 
  %
  %
    \begin{align*}
      \norm{z(t)}\le \norm{z(0)} ~e^{-\sigma t},
    \end{align*}
    then
    \begin{align}\label{necz}
      b_c(t) \ge t~\frac{\Tr(A)+n\sigma}{\ln 2}+n~ \log
      \frac{L}{\norm{z(0)}}.
    \end{align}

  \item If the system   is stabilizable and
    \begin{align*}\label{necx}
      \norm{x(t)} \le \norm{x(0)} ~e^{-\sigma t},
    \end{align*}
    then
    \begin{align}
      b_c(t) \ge t ~\frac{\Tr(A)+n\sigma}{\ln 2}.
    \end{align}
  \end{enumerate}
  In both cases, the information access rate is
  \begin{align}
    R_c > \frac{\Tr(A)+n\sigma}{\ln 2}.
    \label{necrc}
  \end{align}  
\end{theorem}
\vspace{2mm}
\begin{proof}
  Note that~\eqref{necrc} immediately follows by dividing~\eqref{necz}
  and~\eqref{necx} by $t$ and taking the limit for $t \rightarrow
  \infty$. Regarding (i), let us write the solution to~\eqref{syscon}
  as
  \begin{small}
  \begin{align}\label{solutiondif}
    x(t) = e^{At}x(0)+\alpha(t), \quad \alpha(t)=e^{At} \int_{0}^{t}
    {e^{-A\tau}Bu(\tau)d\tau}.
  \end{align}  
  \end{small}
  We then define,
  \begin{align}\label{Arxicgammat}
    \Gamma_t= \{x(t): x(t)=e^{At}x(0)+\alpha(t)~; ~\norm{x(0)} \le
    L \},
  \end{align}
  that is a set which represents the uncertainty at time $t$ given the
  bound $L$ on the norm of the initial condition $x(0)$ and
  $\alpha(t)$. The state of the system can be any point in this
  uncertainty set.  We can find a lower bound on $b_c(t)$ by counting
  the number of balls of radius $\epsilon(t)$, that cover $\Gamma_t$,
  where $\epsilon(t)=\norm{z(0)} ~e^{- \sigma t}$. The Lebesgue
  measure of a sphere of radius $\epsilon$ in $\mathbb{R}^n$ is $k_n
  \epsilon ^n$ where $k_n$ is a constant that changes with dimension.
  Therefore $b_c(t)$, the number of bits of information that the
  controller must have access to by time $t$, should satisfy
  \begin{align*}
    b_c(t)&\ge \log \frac{m(\Gamma_t)}{ m(\mathcal{B} (\epsilon(t)))}
    \\
    &= \log \frac{|\det ((e^A)^t)|m (\norm{x(0)} \le L)}{k_d
      \norm{z(0)}^n ~e^{- n\sigma t}}
    \\
    &=t ~\log |\det (e^A) e^{n\sigma}| + \log \frac{L^n}{\norm{z(0)}^n} \\
    &=t~\log |e^ {\Tr(A)+n\sigma}|+n \log \frac{L}{\norm{z(0)}}.
  \end{align*}
   With access to $b_c(t)$ bits of information, the controller can at
  best identify $x(t)$ up to a ball of radius
  $\epsilon(t)$.  Consequently, (i) follows.
  
  Recall that $\norm{x(0)} \le L$. For any given control trajectory
  $\{u(\tau)\}_{\tau=0}^{\tau=t}$ define
  \begin{align*}
    \Pi_{\{u(\tau)\}_{\tau=0}^{\tau=t}}=\{x(0) :
    \norm{x(t)}<\epsilon(t)\},
  \end{align*}
  where $\epsilon(t)=\norm{x(0)}~ e^{- \sigma t}$. These are the
  sets of all initial conditions for which by choosing the control
  trajectory $\{u(\tau)\}_{\tau=0}^{\tau=t}$, the plant state at time
  $t$, $x(t)$, will be in a ball of radius $\epsilon(t)$.
  $x(t)$ depends linearly on $\{u(\tau)\}_{\tau=0}^{\tau=t}$. As
  a consequence, all of the sets
  $\Pi_{\{u(\tau)\}_{\tau=0}^{\tau=t}}$, are linear transformation of
  each other. The measure of
  $\Pi_{\{u(\tau)=0\}_{\tau \in [0,t]}}$ is 
  \begin{align*}
  |\det(e^{-At})k_n \norm{x(0)}^n e^{-n\sigma t}|=k_n
  \norm{x(0)}^n e^{-(\Tr(A)+n\sigma)t},
  \end{align*}
  which is upper bounded by $k_n
 L^n e^{-(\Tr(A)+n\sigma)t}$. We can then determine a
  lower bound for $b_c(t)$ by counting the number of $\Pi$ sets (for
  different control trajectories $\{u(\tau)\}_{\tau=0}^{\tau=t}$)
  which takes to cover the ball $\norm{x(0)} \le L$.

  Thus, the controller must have access to at least $b_c(t)$ bits by
  time $t$, where
  \begin{align*}
    b_c(t)  &\ge  \log\frac{m(\norm{x(0)}\le L)}{m(\Pi)}\\
    &=  \log\frac{k_n L^n}{k_n \norm{x(0)}^n e^{-(\Tr(A)+n\sigma)t}} \\
    &\ge t~\frac{\Tr(A)+n\sigma}{\ln 2}+n ~\log\frac{L}{L},
  \end{align*}
  and this proves (ii).
\end{proof}
\begin{remark} {\rm Theorem~\ref{thm:necc-access-rate} is
    valid for any control scheme, and the controller does not
    necessarily have to compute the state estimate
    following~\eqref{sysest}. This theorem can be viewed as an extension of the data-rate theorem with exponential convergence guarantees. It states
    that, to have exponential convergence of the estimation error and
    the state, the access rate should be larger than the estimation
    entropy, the latter concept having been recently introduced
    in~\cite{liberzon2016entropy}. A similar
    result  for continuous-time
    systems appears in~\cite{PT-JC:16-tac}, but only for linear
    feedback controllers.  The classic formula of the data-rate
    theorem~\eqref{Datarate}, given in 
    \cite{Mitter,nair2004stabilizability} can be derived as a special case of
    Theorem~\ref{thm:necc-access-rate} by taking $\sigma \rightarrow 0$
    and using continuity. 
  }
  \oprocend
\end{remark}

\section{Necessary condition on the transmission rate for time triggering}

 
We now derive a data-rate theorem for the information transmission rate in two different time-triggered scenarios and  in the presence of unknown communication delays.

In the first scenario, we assume the following time-triggered
implementation:
the sensor transmits at all times $\{t_s^k\}_{k\in \mathbb{N}}$,   where
\begin{align}
\label{TTR2}
t_s^k=kT,
\end{align}
and $T$ denotes the transmission period. 
Note that in this setting,  
the sensor
transmits without considering whether the previous packets have been
received and decoded or not. Consequently, the communication delay is upper bounded as~\eqref{gammma} only when there is not another packet in the communication channel. In this setting, we have the following theorem.
\begin{theorem}\label{thm:necc-TT}
Consider the plant-sensor-channel-controller model described in
  Section~\ref{sec:setup} with plant dynamics~\eqref{syscon}. 
Assume that  the communication delays upper bounded as~\eqref{gammma}
when there is no other packet in the channel, and assuming that the packets are received and decoded by the controller in the
order they are transmitted by the sensor.
Then,  there exists a delay realization 
$\{ \Delta_k
  \}_{k \in \integers}$ such that a rate
\begin{align*}
R_s > \left\{
                \begin{array}{ll}
                  \dfrac{\Tr(A)}{\ln 2}  & \mbox{ if } \gamma < T,\\ \\
                  \dfrac{\Tr(A)\frac{\gamma}{T}}{ \ln 2}& \mbox{ if }  \gamma \ge T. \ \\
                \end{array}
              \right.\\
\end{align*}
is necessary for
asymptotic observability and asymptotic stabilizability.
\end{theorem}
\begin{proof}
  Consider  an observer that can receive the packets transmitted  by the sensor without any delay, and  that has the same knowledge about the system as the controller. 
Let $\zeta_{t}^o$ and $\zeta_{t}^c$   be the uncertainty sets for the state $x(t)$, at the observer and controller, respectively. We have  $\zeta_{0}^o=\zeta_{0}^c$.

We write the solution to~\eqref{syscon}
  as~\eqref{solutiondif}.
  Consequently, we have
   \begin{align*}
    m(\zeta_{t_s^{(k+1)-}}^o)=e^{\Tr(A) T}m(\zeta_{t_s^{k}}^o).
    \end{align*}
 Since the observer receives packets without delay,  we have
\begin{align*}
  m(\zeta_{t_s^{k+1}}^o) \ge \frac{1}{2^{g(t_s^k)}}
  m(\zeta_{t_s^{(k+1)-}}^o) =\frac{1}{2^{g(t_s^k)}} e^{\Tr(A)
    T}m(\zeta_{t_s^{k}}^o).
\end{align*}
By iterating from $k=1$ to $k=\eta$, we have
\begin{align*}
  m(\zeta_{t_s^{\eta-}}^o) \ge
  \dfrac{1}{2^{\sum_{k=1}^{k=\eta-1}g(t_s^k)}}e^{\Tr(A)\eta T}
  m(\zeta_{0}^o).
\end{align*}
However, the controller does not necessarily receive packets
immediately. Indeed, in the worst case,  if $\gamma > T$ the controller receives packets that have been sent in the time interval $[0,\eta T)$  
 by the time
$\eta T+\eta(\gamma-T)=\eta \gamma$.
While, for $T > \gamma$ we have
\begin{align}\label{LALALAND}
  m(\zeta_{t_c^{\eta-}}^c) \ge m(\zeta_{t_s^{\eta-}}^o),
\end{align}
for $T \le \gamma$ we have
\begin{align}\label{LALALAND2}
  \sup_{ \{ \Delta_k \} \leq \gamma } m(\zeta_{t_c^{\eta-}}^c) \ge m(\zeta_{t_s^{\eta-}}^o)e^{\Tr(A)\eta(\gamma-T)}.
\end{align}
It follows that the right-hand side of~\eqref{LALALAND} and~\eqref{LALALAND2} tends to
infinity as $\eta \rightarrow \infty$, making it impossible to stabilize   or
track the state, if
 \begin{align*}
    \infty &= \lim_{\eta \rightarrow \infty}
    \dfrac{1}{2^{\sum_{k=1}^{k=\eta-1}g(t_s^k)}}e^{\Tr(A)\eta T}\\\nonumber
    &=\lim_{\eta \rightarrow \infty}\exp\left\{T \eta  \left(\Tr(A) -\ln 2\frac{\sum_{k=1}^{k=\eta-1}g(t_s^k)}{T\eta}\right)\right\}
      \end{align*}
    for $T>\gamma$,
    and 
    \begin{align*}
    \infty &= \lim_{\eta \rightarrow \infty}
    \dfrac{1}{2^{\sum_{k=1}^{k=\eta-1}g(t_s^k)}}e^{\Tr(A)\eta \gamma}
    \\\nonumber
    &=\lim_{\eta \rightarrow \infty}\exp\left\{T \eta  \left(\Tr(A)\frac{\gamma}{T} -\ln 2\frac{\sum_{k=1}^{k=\eta-1}g(t_s^k)}{T\eta}\right)\right\}
      \end{align*}
      for $T<\gamma$.
The result now follows.
\end{proof}
\begin{remark}
Theorem~\ref{thm:necc-TT} provides a data-rate theorem for the information transmission rate without imposing  exponential convergence guarantees. It shows the existence of    a critical delay value $\gamma =T$, at which the rate begins to increase linearly with the delay. 
\oprocend
\end{remark}

We next
consider a different time-triggered scenario. Let
\begin{align}
  \label{TTR}
  t_s^0 = 0, \quad t_s^{k+1}=t_s^{k}+(\lfloor \Delta_k/T
  \rfloor+1)T,
\end{align}
where $T$ is a fixed non-negative real number. In this case, the sensor
transmits only at integer multiples of the period $T$, after the
previous packet is received. It follows that there is no delay accumulation, and for all packets the delay satisfies \eqref{gammma}. 
In this setting, we have the following result for 
exponential convergence of the
estimation error to zero.
\begin{theorem}\label{thm:necc-TT1111}
Consider the plant-sensor-channel-controller model described in
  Section~\ref{sec:setup} with plant dynamics~\eqref{syscon}, and state estimation error $z (t)$.  Let $\sigma \in \real$ be
  positive.
  If using the time-triggered implementation \eqref{TTR} the state estimation error satisfies
    \begin{align}\label{tt-5d}
    \norm{z(t_s^k)}\le \norm{z(0)} ~e^{-\sigma t_s^k},
  \end{align} 
for all $k\in \integers$,  then there exists a delay realization $\{ \Delta_k
  \}_{k \in \integers}$ which requires
  \begin{align}\label{TTnESGMTRH}
        R_s \ge \frac{(\Tr(A)+n\sigma)(\lfloor \frac{\gamma}{T} \rfloor+1)}{ \ln 2}.  \end{align}
\end{theorem}
\begin{proof}
Using~\eqref{solutiondif}
 we know the state of the system can be any point in $\Gamma_t$, cf.~\eqref{Arxicgammat}, then, we have
  \begin{align*}
 m(\Gamma_{t_c^{k}})=e^{\Tr(A)\Delta_k}m(\Gamma_{t_s^k}),
  \end{align*}
  and
  \begin{align*}
    m(\Gamma_{t_s^{k+1}}) \ge m(\Gamma_{t_c^k})e^{\Tr(A)\left((\lfloor
        \frac{\Delta_k}{T} \rfloor+1)T-\Delta_k\right)}.
  \end{align*}
%
Iterating from $k=0$ to $k=\eta$, we have
  \begin{align*}
    m(\Gamma_{t_s^{\eta-}}) \ge
    e^{\sum_{k=0}^{k=\eta-1} \Tr(A) (\lfloor
      \frac{\Delta_k}{T} \rfloor+1)T} m(\Gamma_{0})=e^{\Tr(A)t_s^{\eta}}m(\Gamma_{0}).
  \end{align*}
%
%
 We can now obtain a lower bound on $\sum_{k=0}^{k=\eta-1}g(t_s^k)$ by counting
  the number of balls of radius $\norm{z(0)} ~e^{-\sigma t_s^\eta}$, that cover $\Gamma_{t_s^{\eta-}}$. Recall that the Lebesgue measure of a sphere of radius $r$ in $\mathbb{R}^n$ is $k_n  r^n$ where $k_n$ is a constant that depends on the dimension. We have
  \begin{align*}
 \sum_{k=0}^{k=\eta-1}g(t_s^k) &\ge \log\frac{e^{\Tr(A)t_s^\eta}m(\Gamma_{0})}{k_n \norm{z(0)}^n ~e^{-n\sigma t_s^\eta}}\\\nonumber
  &=\log\frac{e^{(\Tr(A)+n\sigma)t_s^\eta}m(\Gamma_{0})}{k_n \norm{z(0)}^n}.
  \end{align*}
  Hence, 
    \begin{align*}
   \sum_{k=0}^{k=\eta-1}g(t_s^k) \geq \log\frac{e^{(\Tr(A)+n\sigma)\eta(\lfloor
            \frac{\gamma}{T} \rfloor+1)T}m(\Gamma_{0})}{k_n \norm{z(0)}^n},
    \end{align*}
 because the sensor, not having any fore-knowledge of the delay, must send at least the number of bits required when $\Delta_k=\gamma$ for all $k\in \integers$, to ensure that 
 \eqref{tt-5d}   holds.
 However, the actual realization of the delay may be $\Delta_k=0$ for all $k \in \integers$, so that we have
\begin{align*}
          R_s \ge \lim_{\eta \rightarrow \infty}\frac{1}{\eta T} \log\dfrac{e^{(\Tr(A)+n\sigma)\eta(\lfloor
                    \frac{\gamma}{T} \rfloor+1)T}m(\Gamma_{0})}{k_n \norm{z(0)}^n},
    \end{align*}
    and the result follows.
\end{proof}
\begin{remark}
In the time-triggered setting governed by~\eqref{TTR2}, a packet is  transmitted without considering whether the previous packets have been
received and decoded.  On the other hand, in the time-triggered setting governed by~\eqref{TTR} a packet is transmitted only after the previous packet is received.  
Letting $\sigma \rightarrow 0$, for $\gamma < T$ both Theorems~\ref{thm:necc-TT} and~\ref{thm:necc-TT1111} reduce to $R_s \ge  \Tr(A)/\ln 2$. Namely, for low values of the delay, and without imposing exponential convergence guarantees, we recover the critical value of the data-rate theorem for the access rate in Theorem~\ref{thm:necc-access-rate}. 
\oprocend
\end{remark}

\begin{remark} {\rm Theorem~\ref{thm:necc-TT} and \ref{thm:necc-TT1111} are
    valid for any control scheme, and the controller does not
    necessarily have to compute the state estimate as~\eqref{sysest}. These theorems can be seen as an extension of the data-rate theorem for the information transmission rate for time-triggered control, with unknown bounded delay.} 
  \oprocend
\end{remark}

\section{Necessary and sufficient conditions on the transmission rate for  event-triggering}\label{sec:HigherDimensions}


\subsection{Component-wise description}\label{HDprelim}

In the proposed event-triggered design, we deal with each coordinate of the system separately. This corresponds to treating the $n$-dimensional system as $n$ scalar, coupled systems. When a triggering occurs for one of the coordinates, the controller should be aware  of which coordinate of the system the received packet corresponds to. Accordingly, we assume that there are $n$ parallel finite-rate digital communication channels between each coordinate of the system and the controller, each subject  to unknown, bounded delay. In the case of a single  communication channel,  we can consider the same triggering strategy, but an additional $\ceil{\log n}$ bits should be appended at the beginning of each packet to identify the coordinate it belongs to.

For deriving our necessary and sufficient conditions for vector system, we assume that all of the eigenvalues of $A$ are real. Recall that every $A \in M_{n,n}(\mathbb{R})$ can be written as $\Phi \Psi
\Phi^{-1}$, where $\Phi$ is a real-valued invertible matrix and
$\Psi=\text{diag}[J_1,\ldots, J_q]$, where each $J_j$ is a Jordan
block corresponding to the real-valued eigenvalue $\lambda_j$ of $A$
\cite{Prasolov}. We let $d_j$ indicate  the order of each $J_j$. 
Without loss of generality assume $A$ is equal to its Jordan block decomposition, that is,
 $A=\text{diag}[J_1,\ldots, J_q]$. 
With the notation of Section~\ref{sec:setup}, we let $x_i^j(t)$, $\hat{x}_i^j(t)$, and $z_i^j(t)$ be the state, state estimation, and state estimation error 
 for the $i^{th}$ coordinate of $j^{th}$ Jordan block, respectively.  
For each
coordinate $i$ of the $j^{th}$ Jordan block we let $\{
t_{s,i}^{k,j} \}_{k \in \integers}$, $\{ t_{c,i}^{k ,j}\}_{k \in \integers}$, $g(t_{s,i}^{k,j})$ be the sequence of transmission times, reception times, and the number of bits that  are transmitted at each triggering time. Similarly,
the $k^{th}$ \emph{communication delay} and $k^{th}$ \emph{triggering interval} can be specified for each coordinate.  Following~\eqref{gammma}, we have
\begin{align}\label{GammaHD}
  \Delta_{k,i}^j= t_{c,i}^{k,j}-t_{s,i}^{k,j} \leq \gamma.
\end{align}
When referring to a generic triggering or reception time, we shall
skip the superscript $k$ in $t_{s,i}^{k,j}$ and $t_{c,i}^{k,j}$.

An event is triggered
for coordinate $i$ in Jordan block $j$ whenever
\begin{subequations}\label{eq:event-trig-strat}
\begin{align}\label{eq:etsHHD}
 |z_i^j(t_{s,i}^j)|=v_i^j(t_{s,i}^j),
\end{align} 
where $v_i^j(t)$ is the
\emph{event-triggering function}
\begin{align}\label{eq:etfHD}
v_i^j(t)=v_{0,i}^j e^{-\sigma t},
\end{align}
where $v_{0,i}^j$ and $\sigma$ are positive real numbers.

Let $\bar{z}_i^j(t_{c,i}^j)$ be an estimate of ${z}_i^j(t_{c,i}^j)$
constructed by the controller knowing 
$|z_i^j(t_{s,i}^j)|=v_i^j(t_{s,i}^j)$, the bound~\eqref{GammaHD}, and the decoded packet received through the communication channel. We define the following updating procedure,
called \textit{jump strategy}
\begin{align}\label{eq:jumpstHD}
 \hat{x}_i^j(t_{c,i}^{j+})=\bar{z}_i^j(t_{c,i}^j)+\hat{x}_i^j(t_{c,i}^j).
\end{align}
Note that with this jump strategy, we have
\begin{align*}
z_i^j(t_{c,i}^{j+})=x_i^j(t_{c,i}^j)-\hat{x}_i^j(t_{c,i}^{j+})=z_i^j(t_{c,i}^j)-\bar{z}_i^j(t_{c,i}^j).
\end{align*}
When a triggering occurs for coordinate $i$ of the $j^{th}$ Jordan block, we assume that the sensor sends enough bits  to ensure
\begin{align}
  |z_i^j(t_{c,i}^{j+})| \le\rho(t_{s,i}^j):= \rho_0  e^{-\sigma \gamma} v(t_{s,i}^{j}).
  \label{eq:jump-upp-HD}
\end{align}
\end{subequations}
When referring to a generic Jordan block, 
we skip the superscript and subscript~$j$.
For the scalar case we skip the subscript~$i$ too.

The transmission rate for each
coordinate is then
\begin{align*}
  R_{s,i} ^j= \limsup_{N_i^j \rightarrow \infty} \frac{\sum_{k=1}^{N_i^j} g(t_{s,i}^{k,j})}{\sum_{k=1}^{N_i^j}
    \Delta'^j_{k,i}}
\end{align*}

Assuming $n$ parallel communication channels
between the plant and the controller,  each devoted to a
coordinate separately, we have
\begin{align*}
R_s=\sum_{j=1}^{q}\sum_{i=1}^{d_j}R_{s,i}^j.
\end{align*}
 To obtain our necessary condition, we need to restrict the class of
allowed quantization policies.  We assume that, at each triggering
event, there exists a delay such that the sensor can reduce the
estimation error at the controller to at most a fraction of the
maximum value $\rho(t_{s,i}^j)$ required by~\eqref{eq:jump-upp-HD}. This is a
natural assumption, see~\cite{OurJournal1}.
\begin{ansatz}\label{Defnition11}
  The controller can only achieve $\nu$-precision quantization, namely
  there exists $\nu\geq1$, and a delay at most $\beta:=\frac{1}{A} \ln
  (1+2 \rho_0 e^{-\sigma \gamma})$, such that
  \begin{align*}
    |z(t_{c,i}^j)-\bar{z}(t_{c,i}^j)|\geq \frac{\rho(t_{s,i}^j)}{\nu}.
  \end{align*} 
\end{ansatz}



\subsection{Review of results in the scalar case}
The following results  for scalar systems are the building blocks for our vector case derivation and appear in~\cite{OurJournal1}.

\begin{theorem}\label{thm:necc-cond-ET}
Consider the plant-sensor-channel-controller model described in
  Section~\ref{sec:setup} with plant dynamics~\eqref{syscon}, estimator   dynamics~\eqref{sysest}, and $n=1$. 
If  using the event-triggering
  strategy~\eqref{eq:event-trig-strat}, packet sizes such that $z(t_c)$ is determined at
  the controller within a ball of radius $\rho(t_{s})=\rho_0
  e^{-\sigma \gamma} v(t_{s})$ with $\nu$-precision, and  the state estimation error satisfies \eqref{tt-5d},
then there exists a delay realization $\{ \Delta_k
  \}_{k \in \integers}$ which requires
\begin{align}\label{necessT}
    R_s \ge \frac{A+\sigma}{\ln\nu+\ln(2+\frac{e^{\sigma
          \gamma}}{\rho_0})}\max \left\{0,\log \frac{(e^{A
          \gamma}-1)}{ \rho_0 e^{-\sigma \gamma} }\right\}.
\end{align}
Moreover, when $\sigma$ is sufficiently large the result can be approximated by 
\begin{align}
\label{approxnec}
  R_s \geq \frac{A+\sigma }{\ln 2}
  \max \left\{0 ,1+\frac{\log (e^{A \gamma}-1)}{-\log (\rho_0
      e^{-\sigma \gamma})}\right\}. 
\end{align}
\end{theorem}

We also have a corresponding sufficient condition  for the scalar case.
\begin{theorem}\label{thm:suf-cond-ET} 
Consider the plant-sensor-channel-controller model described in
  Section~\ref{sec:setup} with plant dynamics~\eqref{syscon}, estimator   dynamics~\eqref{sysest}, and $n=1$.
If   the state estimation error satisfies  $|z(0)|<v_0$,  using the event-triggering
  strategy~\eqref{eq:event-trig-strat} we can achieve
  \begin{align*}
    |z(t)| \le  v_0 e^{(A+\sigma) \gamma} e^{-\sigma t},
  \end{align*}
  with an information transmission rate
  \begin{align}\label{Sufi}
    R_s \ge
    \frac{A+\sigma}{-\ln(\rho_0 e^{-\sigma \gamma})}
    \max\left\{0,1+\log\frac{b\gamma (A+\sigma)}{\ln(1+\rho_0 e^{-
          (\sigma+A) \gamma})}\right\},
  \end{align}
  where $\rho_0$ is a constant in the interval $(0,1)$, and $b>1$.
\end{theorem}


\begin{remark}\label{TTVSET}
{\rm
  %
  %

\begin{figure}[t]
  \centering
  \includegraphics[width=8.2cm]{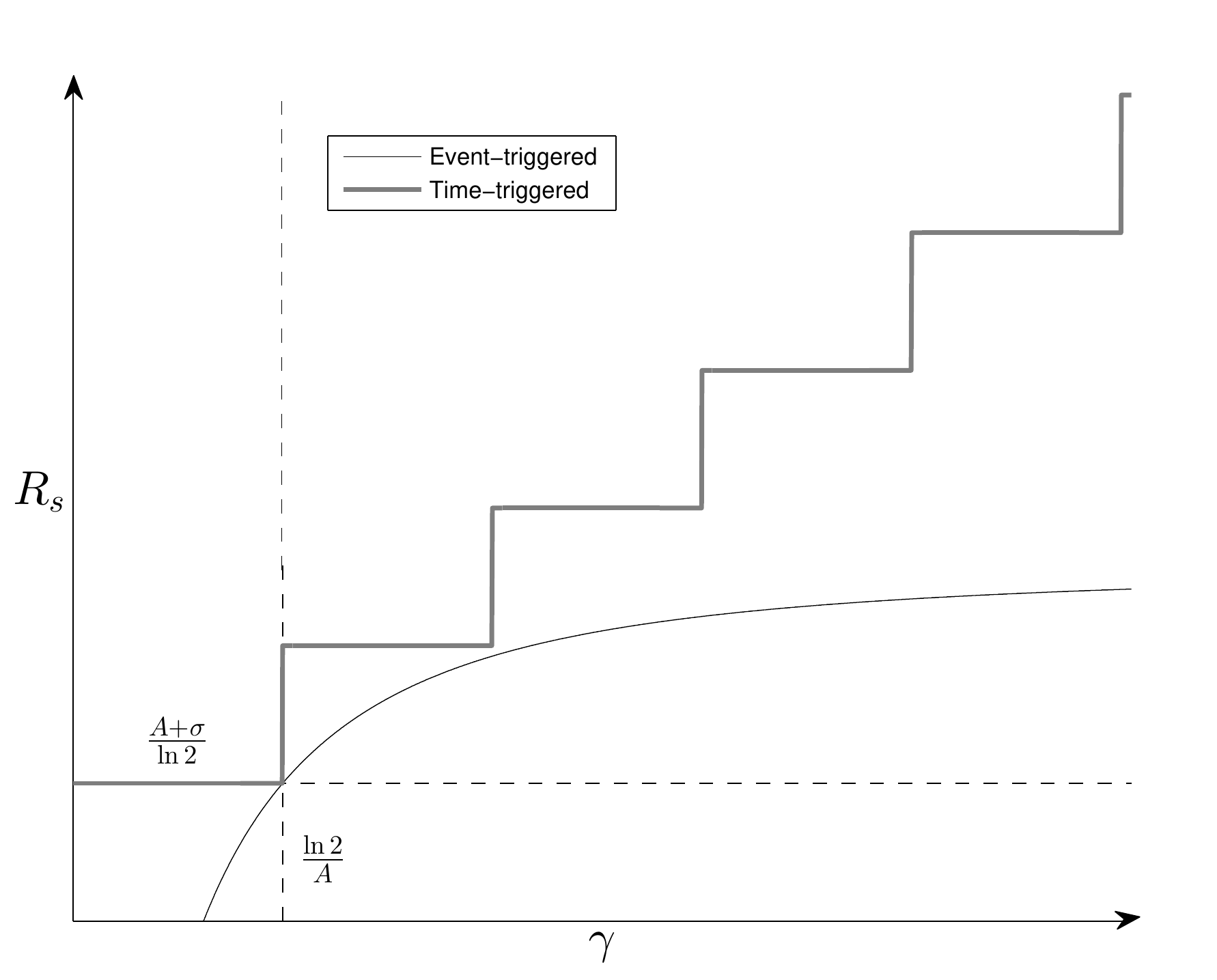}
  \caption{Illustration of the necessary bit rate for time-triggering
    control of a scalar plant~\eqref{TTnESGMTRH} and approximation of
    the necessary bit rate for event-triggering control of a scalar
    plant~\eqref{approxnec} versus the worst-case delay in the
    communication channel. For the time-triggered scheme, $T=\ln
    2/A$.}\label{fig:ness}
  \vspace*{-1ex}
\end{figure}
Figure~\ref{fig:ness} compares the results of Theorem~\ref{thm:necc-TT1111} and Theorem~\ref{thm:necc-cond-ET}.
For small values of $\gamma$,  the necessary  transmission rate in Theorem~\ref{thm:necc-cond-ET} becomes, cf.~\cite{OurJournal1},
   \begin{align}\label{smgaet}
        R_s \ge 0.
  \end{align}
On the other hand,
the result of Theorem \ref{thm:necc-TT1111} in the scalar case and for small values of $\gamma$ can be written as 
  \begin{align}\label{smgatt}
        R_s \ge \frac{A+\sigma}{ \ln 2}.
  \end{align}
Comparing \eqref{smgaet} and \eqref{smgatt},  the value of the intrinsic timing information in communication in an event-triggered design becomes evident. When the delay is small, the timing information carried by the triggering events is substantial and ensures that controller can stabilize the system. In contrast, for small values of the delay the information transmission rate required by a time-triggered implementation equals the information access rate required by the classic data-rate theorem.


For large delay values, it can be easily shown that while both the necessary and sufficient conditions for the event-triggered design in Theorems  \ref{thm:necc-cond-ET} and \ref{thm:suf-cond-ET}   converge to the asymptote $((A+\sigma)/\ln2)(1+A/\sigma)$ as $\gamma \rightarrow \infty$, the time-triggered result in Theorem~\ref{thm:necc-TT1111} for grows linearly as $\gamma \rightarrow \infty$. 
The reason for this difference is that the time-triggered design~\eqref{TTR} depends only on the delay while the event-triggered scheme depends on both state and delay. In both time-triggered and event-triggered schemes the sensor does not have fore-knowledge of the delay, and the sensor  needs to send larger packets when the worst-case delay is larger. 
On the other hand, the triggering rate in the event-triggering case tends to zero as $\gamma$ tends to infinity.
More precisely, using Lemma 3 of~\cite{OurJournal1} in the event-triggering setting  for all of the possible realizations we have 
   \begin{align*}
   t_s^{k+1}-t_s^k \ge \frac{- \ln (\rho_0
      e^{-\sigma \gamma})}{A + \sigma},
  \end{align*}
  which tends to infinity as $\gamma \rightarrow \infty$. 
  In contrast,  in the time-triggered case for delay realization $\Delta_k=0$ for all $k \in \integers$  we have
 \begin{align*}
 t_s^{k+1}-t_s^k= T,
 \end{align*}
and in this case the rate increases linearly with the delay. 
}\oprocend
\end{remark}

\subsection{Necessary and sufficient transmission rate}
We now extend the event-triggering results to the vector~case.
\begin{theorem}\label{thm:Necvector}
Consider the plant-sensor-channel-controller model described in
  Section~\ref{sec:setup}  with plant dynamics~\eqref{syscon}, where all eigenvalues of $A$ are real and $A$ is equal to its Jordan block decomposition, estimator
  dynamics~\eqref{sysest}, event-triggering strategy~\eqref{eq:event-trig-strat}, and packet sizes such that $z_i^j(t_{c,i}^{k,j})$ is determined at
  the controller within a ball of radius $\rho(t_{s,i}^{k,j})=\rho_0
  e^{-\sigma \gamma} v(t_{s,i}^{k,j})$ with $\nu$-precision.
Then there exist a delay realization such that
\begin{align*}
  R_s > \sum_{j=1}^{q}\frac{d_j(\lambda_j+\sigma)}{\ln\nu+\ln(2+\frac{e^{\sigma \gamma}}{\rho_0})}\max \left\{0,\log \frac{(e^{\lambda_j \gamma}-1)}{ \rho_0
      e^{-\sigma \gamma} }\right\}.
\end{align*}
Moreover, when $\sigma$ is sufficiently large the result can be approximated by 
\begin{align*}
   R_s > \sum_{j=1}^{q}\frac{d_j (\lambda_j+\sigma)}{\ln 2} \max
    \left\{0 ,1+\frac{\log (e^{\lambda_j \gamma}-1)}{-\log (\rho_0
        e^{-\sigma \gamma})}\right\}.
\end{align*}

%
%
\end{theorem}

\begin{theorem}
\label{thm:Sufvector}
Consider the plant-sensor-channel-controller model described in
  Section~\ref{sec:setup} with plant dynamics~\eqref{syscon}, where all eigenvalues of $A$ are real and $A$ is equal to its Jordan block decomposition, estimator
  dynamics~\eqref{sysest},  and event-triggering strategy~\eqref{eq:event-trig-strat}. For the $j^{\text{th}}$ Jordan block choose the following sequence of design parameters
  \begin{align*}
   0<\rho_{1}^j<\ldots<\rho_{d_j-1}^j<\rho_{d_j}^j=\rho_0<1.
  \end{align*}
If the state estimation error satisfies $|z_i^j(0)|\le v_{0,i}^j$, then we can achieve  
\begin{align*}
|z_{i}^j(t)|\le v_{0,i}^j
    ((\rho_0-\rho_{i}^j)+e^{(\lambda_j+\sigma)\gamma}) e^{-\sigma t}
\end{align*}
for $i=1,\ldots, d_j$ and $j=1,\ldots,q$,
  with an information transmission rate $R_s$ at least equal to
  \begin{small}
\begin{align*}
  \sum_{j=1}^{j=q}\sum_{i=1}^{i=d_j}\frac{ (\lambda_j+\sigma)}{-\ln(\rho_i^j e^{-\sigma \gamma})} \max\left(0,1+\log\frac{b\gamma (\lambda_j+\sigma)}{\ln(1+\rho_i^j  e^{- (\sigma+\lambda_j) \gamma})}\right),
  \end{align*}
  \end{small}
where
  \begin{align}\label{GRRMARTIN}
 0<v_{0,i}^j \le \frac{v_{0,i-1}^j
      (\lambda_j+\sigma)(\rho_0-\rho_{i}^j)}{((\rho_0-\rho_{i}^j)+e^{(\lambda_j+\sigma)\gamma})(e^{(\lambda_j+\sigma)\gamma}-1)}, 
\end{align}
for $i=2,\ldots,d_j$ and $j=1,\ldots,q$, and $b>1$.
\end{theorem}

\section{Conclusions}\label{sec:conc}
We investigated observability and stabilizability of a continuous-time scalar systems without disturbances in the presence of a finite rate digital communication channel subjected to unknown delay in the feedback loop. 
Our previous results about inherent information in event-triggered strategy have been extended to the vector case and compared with two time-triggered designs.   Open problems for future research include studying the effect of system disturbances and obtaining exponential convergence guarantees for the stabilizability of the system.
\section*{Acknowledgements}
M. J. Khojasteh wishes to thank Dr. Daniel Liberzon, and Mr. Mojtaba Hedayatpour for  helpful discussions. This research was partially supported by NSF award CNS-1446891.

\bibliography{mybib} 
\bibliographystyle{IEEEtran}
 \end{document}